\documentclass[a4paper]{amsart}

\usepackage{amsmath}
\usepackage{amsfonts}
\usepackage{amsthm}
\usepackage{graphicx}
\usepackage{eucal}
\usepackage{amscd}
\usepackage[all,2cell]{xy}
\usepackage{amssymb}
\usepackage{mathrsfs}

 \usepackage{tikz}
 \usetikzlibrary{positioning}
 \tikzset{mynode/.style={draw,circle,inner sep=1pt,outer sep=0pt}}
\usepackage{fancyhdr}
\newtheorem{teo}{Theorem}[section]
\newtheorem{cor}[teo]{Corollary}

\newtheorem{defi}[teo]{Definition}

\newtheorem{prop}[teo]{Proposition}

\newtheorem{remark}[teo]{Remark}

\newcommand{\A}{\ensuremath{\mathbb{A}}}
\newcommand{\Spec}{\operatorname{Spec}}
\newcommand{\End}{\operatorname{End}}

\newcommand{\Cal}[1]{{\mathcal #1}}

\newdir{ |>}{{}*!/-3.5pt/@{|}*!/-8pt/:(1,-.2)@^{>}*!/-8pt/:(1,+.2)@_{>}}

\dedicatory{}

\usepackage{fancyhdr}
\begin{document}
\title[What is the spectral category?]{What is the spectral category?}

\author[M.~J.~Arroyo Paniagua, A.~Facchini, M.~Gran and G.~Janelidze]{Mar\'{i}a Jos\'{e} Arroyo Paniagua}
\address[Mar\'{i}a Jos\'{e} Arroyo Paniagua]{Departamento de Matem\'{a}ticas, Divisi\'{o}n de Ciencias B\'{a}sicas e Ingenier\'{i}a, Universidad Aut\'{o}noma Metropolitana, Unidad Iztapalapa, Mexico, D. F., M\'{e}xico}
\thanks{The second author was partially supported by Ministero dell'Istruzione, dell'Universit\`a e della Ricerca (Progetto di ricerca di rilevante interesse nazionale ``Categories, Algebras: Ring-Theoretical and Homological Approaches (CARTHA)'') and Dipartimento di Matematica ``Tullio Levi-Civita'' of Universit\`a di Padova (Research program DOR1828909 ``Anelli e categorie di moduli''). The third author was supported by a ``Visiting scientist scholarship - year 2019'' by Universit\`a degli Studi di Padova. The fourth author was partially supported by South African NRF}
\email{mariajose.mja@gmail.com}

\author[ ]{Alberto Facchini}
\address[Alberto Facchini]{Dipartimento di Matematica ``Tullio Levi-Civita'', Universit\`{a} di Padova, Padova, Italy}
\thanks{}
\email{facchini@math.unipd.it}

\author[ ]{Marino Gran}
\address[Marino Gran]{Universit\'{e} catholique de Louvain, D\'{e}partement de Math\'{e}matique, 1348 Louvain-la-Neuve, Belgique}
\thanks{}
\email{marino.gran@uclouvain.be}

\author[ ]{George Janelidze}
\address[George Janelidze]{Department of Mathematics and Applied Mathematics, University of Cape Town, Rondebosch 7700, South Africa}
\email{george.janelidze@uct.ac.za}


\subjclass[2010]{}

\begin{abstract}
For a category $\mathcal{C}$ with finite limits and a class $\mathcal{S}$ of monomorphisms in $\mathcal{C}$ that is pullback stable, contains all isomorphisms, is closed under composition, and has the strong left cancellation property, we use pullback stable $\mathcal{S}$-essential monomorphisms in $\mathcal{C}$ to construct a spectral category $\mathrm{Spec}(\mathcal{C},\mathcal{S})$. We show that it has finite limits and that the canonical functor $\mathcal{C}\to \mathrm{Spec}(\mathcal{C},\mathcal{S})$ preserves finite limits. When $\mathcal{C}$ is a normal category, assuming for simplicity that $\mathcal{S}$ is the class of all monomorphisms in $\mathcal{C}$, we show that pullback stable $\mathcal{S}$-essential monomorphisms are the same as what we call subobject-essential monomorphisms.
\end{abstract}

\date{\today}

\maketitle

\section{Introduction}

 The \textit{spectral category} $\mathrm{Spec}(\mathcal{C})$ of a Grothendieck category $\mathcal{C}$ was introduced by Gabriel and Oberst in [11]. According to the Abstract of [11], $\mathrm{Spec}(\mathcal{C})$ is obtained from $\mathcal{C}$ by formally inverting all essential monomorphisms. Although there is no reference to Gabriel and Zisman [12], the definition given in Section 1.2 of [11] is in fact a construction based on the fact that the class of essential monomorphisms in  $\mathcal{C}$ admits the calculus of right fractions. Indeed, it presents the abelian groups $\mathrm{Hom}_\mathrm{Spec({\mathcal{C})}}(A,B)$ (for all $A,B\in \mathrm{Ob}(\mathcal{C})=\mathrm{Ob}(\mathrm{Spec}(\mathcal{C}))$) as directed colimits
\begin{equation*}
\mathrm{Hom}_\mathrm{Spec({\mathcal{C})}}(A,B)=\mathrm{colim}\,\mathrm{Hom}_\mathcal{C}(A',B)
\end{equation*}
taken over all subobjects $A'$ of $A$. It is also easy to see that the spectral category $\mathrm{Spec}(\mathcal{C})$ can equivalently be defined as the quotient category of the category of injective objects in $\mathcal{C}$ modulo the ideal consisting of all morphisms in $\mathcal{C}$ whose kernels are essential monomorphisms. Although this is not mentioned in [11], it is said there that $\mathrm{Spec}(\mathcal{C})$ is a replacement of the \textit{spectrum} of $\mathcal{C}$, which is defined (when $\mathcal{C}$ is the category of modules over a ring) as the collection of isomorphism classes of indecomposable injective objects.

Introducing the spectral category of a Grothendieck category $\mathcal{C}$ can also be motivated by non-functoriality of injective envelopes as follows. For each object $C$ in $\mathcal{C}$, let us fix an injective envelope (=injective hull) $\iota_C:C\to E(C)$ of it. One might expect $E$ to become an endofunctor of $\mathcal{C}$, and $\iota$ to become a natural transformation $1_{\mathcal{C}}\to E$. However, there are strong negative results against these expectations:
\begin{itemize}
	\item According to Proposition 1.12 in [13], $E$ cannot be made a functor even when $\mathcal{C}$ is the category of abelian groups.
	\item Let $R$ be a ring and $\mathcal{C}$ the category of $R$-modules. The ring $R$ can be chosen in such a way that not all $R$-modules are injective, but $E$ can be made an endofunctor of $\mathcal{C}$ (see [13, Exercise 24, p. 48] or [8]), but even in those cases $\iota$ will not become a natural transformation $1_{\mathcal{C}}\to E$. This follows from a very general Theorem 3.2 of [2].   
\end{itemize}
On the other hand, the canonical functor $P:\mathcal{C}\to \mathrm{Spec}(\mathcal{C})$, which the spectral category $\mathrm{Spec}(\mathcal{C})$ comes equipped with, nicely plays the roles of both $1_\mathcal{C}$ and $E$, since each object in $\mathrm{Spec}(\mathcal{C})$ is injective, as shown in [11].

In this paper, however, we are not interested in injective objects, and our main aim is to construct $\mathrm{Spec}(\mathcal{C})$ in full generality, when $\mathcal{C}$ is supposed to be an arbitrary category with finite limits. Apart from the Grothendieck category case above, this was already done in the case of an arbitrary abelian category [12, p. 15], and for some non-additive categories [3].

In fact we begin by taking not just an arbitrary category $\mathcal{C}$ with finite limits, but also any class $\mathcal{S}$ of its monomorphisms that contains all isomorphisms and is pullback stable and closed under composition. We define the spectral category $\mathrm{Spec}(\mathcal{C},\mathcal{S})$ of the pair $(\mathcal{C},\mathcal{S})$ to be the category
\begin{equation*}
\mathcal{C}[(\mathrm{St}(\mathrm{Mono}_E(\mathcal{C},\mathcal{S})))^{-1}]
\end{equation*}
of fractions of $\mathcal{C}$ for the class $\mathrm{St}(\mathrm{Mono}_E(\mathcal{C},\mathcal{S}))$ of pullback stable $\mathcal{S}$-essential monomorphisms of $\mathcal{C}$. When $\mathcal{S}$ is the class of all monomorphisms in $\mathcal{C}$, we write $\mathrm{Spec}(\mathcal{C},\mathcal{S})=\mathrm{Spec}(\mathcal{C})$ and call this category the spectral category of $\mathcal{C}$. 

We make various observations concerning the spans and fractions involved. The most important one is that the class $\mathrm{St}(\mathrm{Mono}_E(\mathcal{C},\mathcal{S}))$ admits the calculus of right fractions, just as the class of essential monomorphism in an abelian category does.

We point out that the spectral category $\mathrm{Spec}(\mathcal{C})$ has finite limits and that the
canonical functor $P:\mathcal{C}\to \mathrm{Spec}(\mathcal{C})$ preserves finite limits. When $\mathcal{C}$ is a normal category [18], assuming for simplicity that $\mathcal{S}$ is the class of all monomorphisms
in $\mathcal{C}$, we show that pullback stable $\mathcal{S}$-essential monomorphisms are the same as
what we call subobject-essential monomorphisms. These are those monomorphisms $m:M\to A$ in $\mathcal{C}$ such that, for any monomorphism $n:N\to A$, one has that $N=0$ whenever $M\times_AN=0$. Finally, when $\mathcal{C}$ is normal, the monoid $\mathrm{End}_{\mathrm{Spec}(\mathcal{C})}(P(A))$ of endomorphisms of an object $P(A)$ in the spectral category is a division monoid whenever $A$ is a \textit{uniform} object (a notion extending the classical one of uniform module in the additive context).

The theory we develop is indeed an extension of what was done in [11] for
the case of Grothendieck categories and in [3] for the category of $G$-groups. Note
that there are several papers involving essential monomorphisms in non-abelian
contexts (see e.g. [4, 21] and the references therein), although it is not their purpose
to introduce spectral categories.\\

\vspace{2mm}

\noindent \textit{Throughout this paper,} $\mathcal{C}$ \textit{denotes a category with finite limits}.

\section{Stabilization of classes of morphisms}

Let $\mathcal{M}$ be a class of morphisms in $\mathcal{C}$. Following \cite{[CJKP]}, define the stabilization $\mathrm{St}(\mathcal{M})$ of $\mathcal{M}$ as the class of morphisms $m\colon M\to A$ such that, for every pullback diagram of the form
\begin{equation*}\xymatrix{U\ar[d]\ar[r]^{u}&X\ar[d]\\M\ar[r]_{m}&A,}
\end{equation*}
$u$ is in $\mathcal{M}$. Let us recall that the symbol ``$\mathrm{St}$'' was used in \cite{[JT]}, while in \cite{[CJKP]} the stabilization of $\mathcal{M}$ was simply denoted by $\mathcal{M}'$. { Similar constructions were also used before, of course. }

\begin{prop}\label{prop-STM}
	The stabilization $\mathrm{St}(\mathcal{M})$ of $\mathcal{M}$ has the following properties:
	\begin{itemize}
		\item [(a)] The class $\mathrm{St}(\mathcal{M})$ is pullback stable.
		\item [(b)] If $\mathcal{M}$ contains all isomorphisms, then so does $\mathrm{St}(\mathcal{M})$.		                                                                                                                                                                                                                                                                                                                                                                                                                                                                                                                                                                                                                                                                          
		\item [(c)] If $\mathcal{M}$ is closed under composition, then so is $\mathrm{St}(\mathcal{M})$.
		\item [(d)] If $\mathcal{M}$ has the right cancellation property of the form
		\begin{equation*} (mm'\in\mathcal{M}\,\,\&\,\,m' \in \mathcal{S})\Rightarrow m\in\mathcal{M}
		\end{equation*}
		for some pullback stable class $\mathcal{S}$ of morphisms in $\mathcal{C}$, 
		then $\mathrm{St}(\mathcal{M})$ has the same property with respect to the same class $\mathcal{S}$.
		\item [(e)] If $\mathcal{M}$ has the weak right cancellation property
		\begin{equation*} (mm'\in\mathcal{M}\,\,\&\,\,m'\in \mathcal{M})\Rightarrow m\in\mathcal{M},
		\end{equation*}
		then $\mathrm{St}(\mathcal{M})$ has the same property.
		\item [(f)] $\mathrm{St}(\mathcal{M})$ has the left cancellation property of the form
		\begin{equation*} (mm'\in\mathrm{St}(\mathcal{M})\,\,\&\,\,m\in\mathrm{Mono}(\mathcal{C}))\Rightarrow m'\in\mathrm{St}(\mathcal{M}),
		\end{equation*}
		where $\mathrm{Mono}(\mathcal{C})$ denotes the class of all monomorphisms in $\mathcal{C}$.
	\end{itemize}
\end{prop}

\begin{proof} (a) and (b) are obvious.
	
	To prove (c), (d), and (e), use a diagram of the form
	\begin{equation*}\xymatrix{U'\ar[d]\ar[r]&U\ar[d]\ar[r]&X\ar[d]\\M'\ar[r]_{m'}&M\ar[r]_m&A,}
	\end{equation*}
	where the squares are pullbacks and the unlabeled arrows are the suitable pullback projections.
	
	To prove (f), consider the diagram
	\begin{equation*}\xymatrix{M'\times_ML\ar[d]\ar[r]&L\ar[d]_l\ar[r]^{1_L}&L\ar[d]^l\\M'\ar[d]_{1_{M'}}\ar[r]^{m'}&M\ar[d]_{1_{M}}\ar[r]^{1_M}&M\ar[d]^m\\M'\ar[r]_{m'}&M\ar[r]_m&A,}
	\end{equation*}
	where $l\colon L\to M$ is an arbitrary morphism and the unlabeled arrows are the pullback projections. Note that all its squares are pullbacks, except for the right-hand bottom square, although it is also a pullback if $m$ is a monomorphism. Therefore, if $mm'$ is in $\mathrm{St}(\mathcal{M})$ and $m$ is a monomorphism, the pullback projection $M'\times_ML\to L$ is in $\mathcal{M}$. This proves the desired implication. 
\end{proof}
\begin{remark} {\em The properties \ref{prop-STM}(a)-(c) are mentioned in \cite{[CJKP]} and \ref{prop-STM}(d) is `almost' there, with $\mathcal{E}$ instead of $\mathcal{M}$. Property \ref{prop-STM}(e) also holds in the main example there, but for the trivial reason that $(mm'\in\mathrm{St}(\mathcal{E})\,\,\&\,\,m\in\mathrm{Mono}(\mathcal{C}))$ implies that $m$ is an isomorphism.} \end{remark}

\section{Essential and pullback stable essential monomorphisms}

{\em {Throughout this paper}, we will consider a class $\mathcal{S}$ of monomorphisms in $\mathcal{C}$ that is pullback stable, contains all isomorphisms, is closed under composition, and has the strong left cancellation property
\begin{equation*} mm'\in\mathcal{S}\Rightarrow m'\in\mathcal{S}.
\end{equation*}}
According to a well-known definition, a morphism $m\colon M\to A$ from $\mathcal{S}$ is said to be an $\mathcal{S}$-\textit{essential monomorphism}, if a morphism $f\colon A\to B$ from $\mathcal{C}$ is in $\mathcal{S}$ whenever so is $fm$. When $\mathcal{S}$ is the class of all monomorphisms in $\mathcal{C}$, we will say ``essential'' instead of ``$\mathcal{S}$-essential''. The class of all $\mathcal{S}$-essential monomorphisms will be denoted by $\mathrm{Mono}_E(\mathcal{C},\mathcal{S})$. This class has many ``good'' properties well-known in the case of an abelian $\mathcal{C}$ with $\mathcal{S}$ being the class of all monomorphisms in $\mathcal{C}$ (see e.g.~any of the following: Section 5 in Chapter II of \cite{[Ga]}, Section 2 in Chapter III of \cite{[Mi]}, or Section 15.2 of \cite{[S]}), and also known in the general case, as briefly mentioned in Remark 9.23 of \cite{[AHS]}. The known properties we will need are collected in:   
\begin{prop}\label{Prop-Mono-ECS}
	The class $\mathrm{Mono}_E(\mathcal{C},\mathcal{S})$ of $\mathcal{S}$-essential monomorphisms 	\begin{itemize}
		\item [(a)] contains all isomorphisms;	                                                                                                                                                                                                                                                                                                                                                                                                                                                                                                                                                                                                                                                                          
		\item [(b)] is closed under composition;
		\item [(c)]  has the right cancellation property of the form
		\begin{equation*} (mm'\in\mathrm{Mono}_E(\mathcal{C},\mathcal{S})\,\,\&\,\,m\in \mathcal{S})\Rightarrow m\in\mathrm{Mono}_E(\mathcal{C},\mathcal{S});
		\end{equation*}
		\item [(d)] has the weak right cancellation property 
		\begin{equation*} (mm'\in\mathrm{Mono}_E(\mathcal{C},\mathcal{S})\,\,\&\,\,m'\in\mathrm{Mono}_E(\mathcal{C},\mathcal{S}))\Rightarrow m\in\mathrm{Mono}_E(\mathcal{C},\mathcal{S})
		\end{equation*}
		and, in particular, every split monomorphism that belongs to {$\mathrm{Mono}_E(\mathcal{C},\mathcal{S})$} is an isomorphism.\qed
	\end{itemize}
\end{prop}
\begin{remark} {\em Note the difference between our Proposition 3.1(c) and Proposition 9.14(3) of \cite{[AHS]}: we have omitted the redundant assumption $m'\in\mathcal{S}$.
} \end{remark}

From Propositions 2.1 and 3.1, we immediately obtain:
\begin{teo}\label{ST-Mono} 
	The class $\mathrm{St}(\mathrm{Mono}_E(\mathcal{C}, \mathcal{S}))$ of pullback stable  $\mathcal{S}$-essential monomorphisms in $\mathcal{C}$ 
	\begin{itemize}
		\item [(a)] is pullback stable;
		\item [(b)] contains all isomorphisms;		                                                                                                                                                                                                                                                                                                                                                                                                                                                                                                                                                                            
		\item [(c)] is closed under composition;
		\item [(d)] has the right cancellation property of the form
		\begin{equation*} (mm'\in \mathrm{St}(\mathrm{Mono}_E(\mathcal{C},\mathcal{S}))\,\,\&\,\,m\in \mathcal{S})\Rightarrow m\in \mathrm{St}(\mathrm{Mono}_E(\mathcal{C},\mathcal{S}));
		\end{equation*}
		\item [(e)] has the weak right cancellation property 
		\begin{equation*} (mm'\in\mathrm{St}(\mathrm{Mono}_E(\mathcal{C},\mathcal{S}))\,\,\&\,\,m'\in\mathrm{St}(\mathrm{Mono}_E(\mathcal{C},\mathcal{S})))\Rightarrow m\in\mathrm{St}(\mathrm{Mono}_E(\mathcal{C},\mathcal{S})),
		\end{equation*}
		and, in particular, every split monomorphism that belongs to $\mathrm{St}(\mathrm{Mono}_E(\mathcal{C},\mathcal{S}))$ is an isomorphism;
		\item [(f)] has the left cancellation property of the form
		\begin{equation*} (mm'\in\mathrm{St}(\mathrm{Mono}_E(\mathcal{C},\mathcal{S}))\,\,\&\,\,m\in \mathrm{Mono}(\mathcal{C}))\Rightarrow m'\in\mathrm{St}(\mathrm{Mono}_E(\mathcal{C},\mathcal{S})).\qed
		\end{equation*}
	\end{itemize}
\end{teo}

\section{Spans and fractions}

Let $\mathcal{C}$ be a category with pullbacks. The bicategory $\mathrm{Span}(\mathcal{C})$ of spans in $\mathcal{C}$, originally introduced in \cite{[B1]} (motivated by the study of spans of additive categories in \cite{[Y]}) is constructed as follows, omitting obvious coherent isomorphisms:
\begin{itemize}
	\item The objects (=0-cells) of $\mathrm{Span}(\mathcal{C})$ are the same as the objects of $\mathcal{C}$.
	\item A morphism (1-cell) $A\to B$ in $\mathrm{Span}(\mathcal{C})$ is a diagram in $\mathcal{C}$ of the form
	\begin{equation*}\xymatrix{A&X\ar[l]_x\ar[r]^f&B,} 
	\end{equation*} 
	usually written either as the triple $(f,X,x)$ or as the pair $(f,x)$.
	\item The composite $(g,Y,y)(f,X,x)=(gq,X\times_BY,xp)$ of $(f,X,x)\colon A\to B$ and $(g,Y,y)\colon B\to C$ is defined via the diagram
	\begin{equation*}\xymatrix{&&X\times_BY\ar[dl]_p\ar[dr]^q\\&X\ar[dl]_x\ar[dr]^f&&Y\ar[dl]_y\ar[dr]^g\\A&&B&&C} 
	\end{equation*}
	in which $p\colon X\times_BY\to X$ and $q\colon X\times_BY\to Y$ are the pullback projections. 
	\item A 2-cell from $(f,X,x)\colon A\to B$ to $(f',X',x')\colon A\to B$ is a morphism $s\colon X\to X'$ with $x's=x$ and $f's=f$, and the 2-cells compose as in $\mathcal{C}$.
\end{itemize}

More generally, {given} a pullback stable class $\mathcal{M}$ of morphisms in $\mathcal{C}$ that contains all identity morphisms and is closed under composition\, \textendash\, we can then form the bicategory $\mathrm{Span}_{\mathcal{M}}(\mathcal{C})$ as above but requiring its morphisms $(f,x)$ to have $x$ in $\mathcal{M}$. As it was observed in a discussion with S. Mac Lane \cite{[JM]} (and most probably known before, which is why the content of that discussion was never published), the assignment $(cls(f,x)\colon A\to B)\mapsto(fx^{-1}\colon A\to B)$ (here $cls$ is the abbreviation for ``class'') determines an isomorphism 
\begin{equation*}
	\Pi(\mathrm{Span}_{\mathcal{M}}(\mathcal{C}))\xrightarrow\approx\mathcal{C}[\mathcal{M}^{-1}],
\end{equation*}
in which:
\begin{itemize}
	\item $\Pi(\mathrm{Span}_{\mathcal{M}}(\mathcal{C}))$ is the Poincar\'e category  of $\mathrm{Span}_{\mathcal{M}}(\mathcal{C})$ { (in the sense of \cite{[B1]})}, that is, it has the same objects as $\mathrm{Span}_{\mathcal{M}}(\mathcal{C})$, and its $\hom$ sets are the sets of connected components of $\hom$ categories of $\mathrm{Span}_{\mathcal{M}}(\mathcal{C})$.
	\item $\mathcal{C}[\mathcal{M}^{-1}]$ is the category of fractions \cite{[GZ]} of $\mathcal{C}$ for $\mathcal{M}$.
\end{itemize}
We are assuming that the reader is familiar with the content of \cite{[B1]}. Repeating here the necessary details from that influential paper would take too much space.

Under the isomorphism above, the functor $\mathcal{C}\to \Pi(\mathrm{Span}_{\mathcal{M}}(\mathcal{C}))$, corresponding to the canonical functor
\begin{equation*}
P_{\mathcal{M}}\colon \mathcal{C}\to \mathcal{C}[\mathcal{M}^{-1}],
\end{equation*} 
is defined by $(f\colon A\to B)\mapsto(cls(f,1_A)\colon A\to B)$.

\bigskip

Recall that a class $\mathcal{M}$ is {\em focal} \cite{[B2]} if it satisfies the following four conditions:
\begin{itemize} \item [$(F_0)$] 
For each object $X\in\Cal C$ there exists an $s\in \mathcal{M}$ with codomain $X$.
 \item [$(F_1)$] For all $\xymatrix{ {} \ar[r]^{s_1}
& {} \ar[r]^{s_0} & }$
with $s_i\in \mathcal{M}$, there exists a morphism $f$ in $\Cal C$ such that the
composite $s_0s_1f$ is defined and is in $\mathcal{M}$.
 \item [$(F_2)$]  Each diagram  
$$\xymatrix{
& \ar[d]^{s}  \\
\ar[r]_f & }$$
with $s\in \mathcal{M}$ can be completed in a commutative
square $$\xymatrix{{} \ar[d]_{s'} \ar[r]^{f'}
& \ar[d]^{s}  \\
\ar[r]_f & }$$ where $s'\in \mathcal{M}$.
 \item [$(F_3)$] If a pair $(f,g)$ of parallel morphisms is coequalized by some $s\in \mathcal{M}$, it is also
equalized by some $s'\in \mathcal{M}$.
\end{itemize}

\begin{prop}\label{focal}
	{ If $\mathcal{M}$ is a pullback stable class of morphisms in $\Cal C$ that contains all identity morphisms and is closed under composition, with $\mathcal{M}\subseteq\mathrm{Mono}(\mathcal{C})$,} then $\mathcal{M}$ is focal and, moreover, $\mathcal{M}$ admits the calculus of right fractions in the sense of \cite{[GZ]}.
\end{prop}
\begin{proof}
	All we need to check is that $\mathcal{M}$ satisfies the condition dual to condition 2.2(d) in Chapter I of \cite{[GZ]}, i.e., that whenever two parallel morphisms $f$ and $g$ admit a morphism $m\in\mathcal{M}$ with $mf=mg$, they also admit a morphism $n\in\mathcal{M}$ with $fn=gn$. This condition holds trivially because $\mathcal{M}\subseteq\mathrm{Mono}(\mathcal{C})$.
\end{proof}

\begin{remark} {\em Note the following levels of generality (in fact there are many more of them, including those suggested by distinguishing sets of morphisms from proper classes of morphisms), where we omitted all required conditions on $\mathcal{M}$ in the first five items:
		\begin{itemize}
			\item [(a)] For an arbitrary class $\mathcal{M}$ of morphisms of $\Cal C$, we can still form the category $\mathcal{C}[\mathcal{M}^{-1}]$ of fractions of $\mathcal{C}$ for $\mathcal{M}$.
			\item [(b)] As shown in \cite{[B2]}, the morphisms of $\mathcal{C}[\mathcal{M}^{-1}]$ can be presented in the form $P_{\mathcal{M}}(f)P_{\mathcal{M}}(x)^{-1}$ with $x\in\mathcal{M}$ if and only if $\mathcal{M}$ satisfies conditions $(F_0)$, $(F_1)$, and $(F_2)$.
			\item [(c)] In particular, this is the case when $\mathcal{M}$ satisfies the conditions dual to conditions 2.2(a), 2.2(b), and 2.2(c) in Chapter I of \cite{[GZ]}.
			\item [(d)] If the equivalent conditions in (b) hold, then the following conditions are equivalent: $\mathrm{(d_1)}$ $\mathcal{M}$ is focal; $\mathrm{(d_2)}$ $\mathcal{M}$ satisfies condition $(F_3)$, which the same as condition 2.2(d) in Chapter I of \cite{[GZ]}; $\mathrm{(d_3)}$ not only can the morphisms of $\mathcal{C}[\mathcal{M}^{-1}]$ be presented as in (b), but also $P_{\mathcal{M}}(f)P_{\mathcal{M}}(x)^{-1}=P_{\mathcal{M}}(f')P_{\mathcal{M}}(x')^{-1}$ if and only there exists a commutative diagram in $\mathcal{C}$ of the form
			\begin{equation*}\xymatrix{&&X\ar[dl]_x\ar[dr]^f\\&A&Y\ar[u]^u\ar[d]_v&B\\&&X'\ar[ul]^{x'}\ar[ur]_{f'}} 
			\end{equation*}
			with $xu\in\mathcal{M}$.
			\item [(e)] In particular, the equivalent conditions $\mathrm{(d_1)}$-$\mathrm{(d_3)}$ hold when the class $\mathcal{M}$ admits the calculus of right fractions in the sense of \cite{[GZ]}.
			\item [(f)] If $\mathcal{M}$ contains all identity morphisms, is closed under composition, and is pullback stable, then not only are we in the situation (c), but we also have the isomorphism between $\mathcal{C}[\mathcal{M}^{-1}]$ and $\Pi(\mathrm{Span}_{\mathcal{M}}(\mathcal{C}))$ mentioned above.
			\item [(g)] The situation of Proposition \ref{focal}. Note, in particular, that in this case the morphisms $u$ and $v$ in the diamond diagram of (d) belong to $\mathcal{M}$. This follows from Proposition 2.1(f) and the fact that $\mathrm{St}(\mathcal{M})=\mathcal{M}$ here. 
		\end{itemize}
The levels of generality listed above are related as follows:
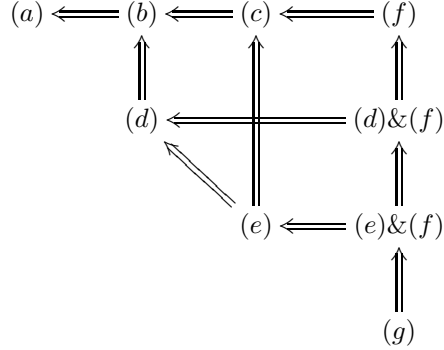
\begin{equation*}\xymatrix{(a)&(b)\ar@{=>}[l]&(c)\ar@{=>}[l]&(f)\ar@{=>}[l]\\&(d)\ar@{=>}[u]&&(d)\&(f)\ar@{=>}[ll]\ar@{=>}[u]\\&&(e)\ar@{=>}[lu]\ar@{=>}[uu]&(e)\&(f)\ar@{=>}[l]\ar@{=>}[u]\\&&&(g)\ar@{=>}[u]} 
\end{equation*}}
\end{remark}

\begin{remark}\label{GabrielZ}
	{\em We recall from \cite{[GZ]} that already in the situation (d), the equivalent conditions mentioned there imply that the $\hom$ sets of $\mathcal{C}[\mathcal{M}^{-1}]$ can be constructed as filtered colimits
	\begin{equation*}{\hom}_{\mathcal{C}[\mathcal{M}^{-1}]}(A,B)=\mathrm{colim}({\hom}(M,B)),
	\end{equation*}
	where the colimit is taken over all $m\colon M\to A$ in $\mathcal{M}$ (see Page 13 in \cite{[GZ]}, where the dual construction is described esplicitly).} 
\end{remark}

\section{The spectral category}

Let $\mathcal{S}$ be a class of monomorphisms in $\mathcal{C}$ satisfying the conditions required at the beginning of Section 3. Then, as follows from (a)-(c) and (f) of Proposition 3.3, the class $\mathcal{M}=\mathrm{St}(\mathrm{Mono}_E(\mathcal{C},\mathcal{S}))$ satisfies the conditions required in Proposition 4.1. We are ready to give the following
\begin{defi}
	\emph{The {\em spectral category} $\mathrm{Spec}(\mathcal{C},\mathcal{S})$ of $(\mathcal{C},\mathcal{S})$ is the  category 
	\begin{equation*}
	\mathcal{C}[(\mathrm{St}(\mathrm{Mono}_E(\mathcal{C},\mathcal{S})))^{-1}]
	\end{equation*}
	of fractions of $\mathcal{C}$ for the class $\mathrm{St}(\mathrm{Mono}_E(\mathcal{C},\mathcal{S}))$ of pullback stable $S$-essential mono-morphisms of $\mathcal{C}$. When $\mathcal{S}$ is the class of all monomorphisms in $\mathcal{C}$, we shall simply write $\mathrm{Mono}_E(\mathcal{C},\mathcal{S})=\mathrm{Mono}_E(\mathcal{C})$ and $\mathrm{Spec}(\mathcal{C},\mathcal{S})=\mathrm{Spec}(\mathcal{C})$, and say that $\mathrm{Spec}(\mathcal{C})$ is the spectral category of $\mathcal{C}$.}
\end{defi}

Thanks to the results of \cite{[GZ]}, specifically Proposition 3.1 and Corollary 3.2 of Chapter I there, our Proposition \ref{focal} implies:
\begin{teo}
	The spectral category $\mathrm{Spec}(\mathcal{C},\mathcal{S})$ has finite limits. Moreover, the canonical functor
	\begin{equation*}\label{canonical-functor}
	P_{\mathcal{C},\mathcal{S}}=P_{\mathrm{St}(\mathrm{Mono}_E(\mathcal{C},\mathcal{S}))}:\mathcal{C}\to \mathrm{Spec}(\mathcal{C},\mathcal{S}),
	\end{equation*} 
	defined by $(f:A\to B)\mapsto(cls(f,1_A):A\to B)$, preserves finite limits.\qed
\end{teo}

\section{Subobject-essential monomorphisms}

Assuming $\mathcal{C}$ to be pointed, we define:

\begin{defi}\emph{
	A monomorphism $m:M\to A$ in $\mathcal{C}$ is said to be {\em subobject-essential} if, for a monomorphism $n:N\to A$, one has $M\times_AN=0\Rightarrow N=0$. The class of all subobject-essential monomorphisms in $\mathcal{C}$ will be denoted by $\mathrm{Mono}_{SE}(\mathcal{C})$.}
\end{defi}
\noindent 
Recall that a regular epimorphism in a category $\mathcal C$ is a morphism that is the coequalizer of two morphisms in $\mathcal C$.

A finitely complete category $\mathcal C$ is \emph{regular} if any morphism $f \colon A \rightarrow B$ can be factorized as the composite morphism of a regular epimorphism $p \colon A \rightarrow I$ and a monomorphism $m \colon I \rightarrow B$
$$\xymatrix{A \ar[rr]^{f} \ar@{>>}[dr]_p  & &  B \\
&{I\,\,  } \ar@{>->}[ur]_m & 
}$$
and these factorizations are pullback-stable.
Following \cite{[Ja]}, we will call $\mathcal{C}$ \textit{normal} if it is pointed, regular, and any regular epimorphism is a normal epimorphism (i.e., a cokernel of some arrow in $\mathcal C$). In such a category, any regular epimorphism is then the cokernel of its kernel and, as a consequence,
 a morphism in $\mathcal{C}$ is a monomorphism if and only if its kernel is zero. 
 \begin{remark}
 \emph{For a pointed variety $\mathcal V$ of universal algebras, being a normal category is the same as being a $0$-regular variety in the sense of  \cite{Fichtner} (see \cite{JMU} for further explanations and historical remarks about the relationship between the properties of $0$-regularity and normality). The algebraic theory of a pointed $0$-regular variety $\mathcal V$ is characterized by the existence of a unique constant $0$ and binary terms $d_1$, ..., $d_n$ such that the identities $d_i(x,x) = 0$ (for $i \in \{ 1, \cdots, n \}$) and the implication $(d_1(x,y)=0 \, \&\, \dots \,\&\, d_n(x,y)=0) \, \Rightarrow x=y$ hold. Intuitively, these operations $d_i(x,y)$ can then be thought of as a kind of ``generalized subtraction''. This implies that the varieties of groups, loops, rings, associative algebras, Lie algebras, crossed modules and $G$-groups (for a group $G$) are all normal. There are also plenty of examples of normal categories that are not varieties, such as the categories of topological groups, cocommutative $K$-Hopf algebras over a field $K$, and $C^*$-algebras, for instance. In general, any semi-abelian category \cite{JMT} is in particular a normal category.}
 \end{remark}

 For an object $A$ in $\mathcal{C}$, the smallest and the largest congruence (=effective equivalence relation) on $A$ will be denoted by $\Delta_A$ and $\nabla_A$, respectively. Note that equalities like $E=\Delta_A$ should usually be understood as equalities of subobjects (of $A\times A$ in this case).

When $\mathcal{C}$ is normal, it is natural to ask how different subobject-essential monomorphisms are from essential ones (recall that  ``essential'' means $\mathcal{S}$-essential for $\mathcal{S}=\mathrm{Mono}(\mathcal{C})$). Most of this section is devoted to studying various ways to compare them. 

Let us begin with the following proposition, well known in the case of an abelian category $\mathcal{C}$: 

\begin{prop}\label{char-essential} 
	If $\mathcal{C}$ is normal, then the following conditions on a monomorphism $m:M\to A$ in $\mathcal{C}$ are equivalent:
	\begin{itemize}
		\item [(a)] $m$ is an essential monomorphism, that is, a morphism $f:A\to B$ is a monomorphism whenever so is $fm$.                                                                                                                                                                                                                                                                                                                                                                                                                                                                                                                                                                                                                                                                                      
		\item [(b)] For any congruence $E$ on $A$, one has $(M\times M)\times_{A\times A}E=\Delta_M\Rightarrow E=\Delta_A$.
		\item [(c)] For any normal monomorphism $n:N\to A$, one has $M\times_AN=0\Rightarrow N=0$.
		\item [(d)] For any morphism $f:A\to B$, one has $ \mathsf{Ker}(fm)=0\Rightarrow \mathsf{Ker}(f)=0$.\qed
	\end{itemize}
\end{prop}
\begin{proof}
(a)${}\Rightarrow {}$(b). Let $(E, e_1, e_2)$ be a congruence on $A$, and $f \colon A \rightarrow C$ a morphism such that $E$ is the kernel pair of $f$:
\begin{equation}\label{congruence}
\xymatrix{E \ar[r]^{e_2} \ar[d]_{e_1} & A \ar[d]^f \\
A \ar[r]_f & C.
}
\end{equation}
Consider the commutative diagram 
$$
\xymatrix{M \ar[r] \ar[d]_{(1_M, 1_M)} & E \ar[r] \ar[d]_{(e_1, e_2)} & C \ar[d]^{(1_C, 1_C)} \\
M \times M \ar[r]_{m \times m} & A \times A \ar[r]_{f \times f} & C \times C}
$$
where the right-hand square is a pullback by definition of kernel pair, and the left-hand square is a pullback by the assumption 
$(M\times M)\times_{A\times A}E=\Delta_M$. The fact that the rectangle is a pullback means that $fm$ is a monomorphism. Since $m$ is an essential monomorphism, it follows that $f$ is a monomorphism and $E = \Delta_A$.

(b)${} \Rightarrow {}$(c) This follows from the fact that a congruence $(E, e_1, e_2)$ as in \eqref{congruence} is the discrete equivalence relation $\Delta_A$ if and only if the normal monomorphism $\ker(f) \colon N \rightarrow A$ corresponding to $E$ is $0 \rightarrow A$.

(c)${} \Rightarrow {}$(d) It suffices to apply the assumption to the pullback
$$
\xymatrix{ {0 = \mathsf{Ker}(fm)} \ar[r] \ar[d] & \mathsf{Ker}(f) \ar[d] \\
M \ar[r]_m & A.
}
$$

(d)${} \Rightarrow {}$(a) This is immediate since, in a normal category, monomorphisms are characterized by the fact that their kernel is $0$.
\end{proof}

From Proposition \ref{char-essential}, we immediately obtain:

\begin{cor}\label{coro-sub-essent}
	Let $\mathcal{C}$ be a normal category. Then: 
	\begin{itemize}
		\item [(a)] Every subobject-essential monomorphism is essential.
		\item [(b)] If $A$ is an object in $\mathcal{C}$ for which every monomorphism with codomain $A$ is normal, then a monomorphism $m:M\to A$ is subobject-essential if and only if it is essential. 
		\item [(c)] In particular, if $\mathcal{C}$ is abelian, then a monomorphism in $\mathcal{C}$ is subobject-essential if and only if it is essential.  
	\end{itemize}
\end{cor}

Next, we have:

\begin{prop}\label{subobj-essential}
	The class $\mathrm{Mono}_{SE}(\mathcal{C})$ of subobject-essential monomorphisms in $\mathcal{C}$ 
	\begin{itemize}
		\item [(a)] contains all isomorphisms;
		\item [(b)]  is closed under composition;
		\item [(c)] has the right cancellation property of the form
		\begin{equation*}
		(mm'\in\mathrm{Mono}_{SE}(\mathcal{C})\,\,\&\,\,m\in\mathrm{Mono}(\mathcal{C}))\Rightarrow m\in\mathrm{Mono}_{SE}(\mathcal{C}).
		\end{equation*}		
		\item [(d)] If $\mathcal{C}$ is normal, then the class $\mathrm{Mono}_{SE}(\mathcal{C})$ has the weak right cancellation property 
		\begin{equation*} (mm'\in\mathrm{Mono}_{SE}(\mathcal{C})\,\,\&\,\,m'\in\mathrm{Mono}_{SE}(\mathcal{C}))\Rightarrow m\in\mathrm{Mono}_{SE}(\mathcal{C})
		\end{equation*}
		and, in particular, every split monomorphism that belongs to it is an isomorphism.
		\item [(e)] It has the left cancellation property of the form
		\begin{equation*}
		(mm'\in\mathrm{Mono}_{SE}(\mathcal{C})\,\,\&\,\,m\in \mathrm{Mono}(\mathcal{C}))\Rightarrow m'\in\mathrm{Mono}_{SE}(\mathcal{C}).
		\end{equation*}
		\item [(f)] If $\mathcal{C}$ is normal, then the class $\mathrm{Mono}_{SE}(\mathcal{C})$ is pullback stable. 
	\end{itemize}	 
\end{prop}
\begin{proof} (a) is obvious.
	
(b) and (c): Given monomorphisms $m\colon M\to A$, $m'\colon M'\to M$ and $n\colon N\to A$, consider the diagram 
	\begin{equation*}\xymatrix{M'\times_AN\ar[d]\ar[r]^{m'\times1}&M\times_AN\ar[d]\ar[r]&N\ar[d]^n\\M'\ar[r]_{m'}&M\ar[r]_m&A,}
	\end{equation*}
where the unlabeled arrows are the suitable pullback projections. Since both squares in this diagram are pullbacks, we can argue as follows: 
\begin{itemize}
	\item If $m, m'\in \mathcal{M}$, then $M'\times_AN=0\Rightarrow M\times_AN=0\Rightarrow N=0$.
	\item If $mm'\in \mathcal{M}$, then $M\times_AN=0\Rightarrow M'\times_AN=0\Rightarrow N=0$, where the first implication holds because $\mathsf{Ker}(m')=0$.
\end{itemize}

(d): Suppose $mm'$ and $m$ are in $\mathrm{Mono}_{SE}(\mathcal{C})$. Thanks to (c), we only need to prove that $m$ is a monomorphism. Therefore, since $\mathcal{C}$ is normal, it suffices to prove that $m$ has zero kernel. For, consider the diagram
\begin{equation*}\xymatrix{0\ar[d]\ar[r]&\mathsf{Ker}(m)\ar[d]^{\ker(m)}\ar[r]&0\ar[d]\\M'\ar[r]_{m'}&M\ar[r]_m&A}
\end{equation*}
and observe that:
\begin{itemize}
	\item Its left-hand square is a pullback because so is its right-hand square, and $mm'$ is a monomorphism because it is in $\mathrm{Mono}_{SE}(\mathcal{C})$.
	\item Since $m'$ is in $\mathrm{Mono}_{SE}(\mathcal{C})$, we have that $\mathsf{Ker}(m)=0$. 
\end{itemize}

(e): Suppose $mm'$ is in $\mathrm{Mono}_{SE}(\mathcal{C})$ and $m$ is a monomorphism. First notice that, since $mm'$ is a monomorphism, so is $m'$. After that, consider the diagram 
\begin{equation*}\xymatrix{M'\times_ML\ar[d]\ar[r]&L\ar[d]_l\ar[r]^{1_L}&L\ar[d]^{ml}\\M'\ar[r]_{m'}&M\ar[r]_m&A,}
\end{equation*}
where the unlabeled arrows are the suitable pullback projections. Since both squares in this diagram are pullbacks, we have
\begin{equation*}
M'\times_ML=0\Rightarrow M'\times_AL=0\Rightarrow L=0.
\end{equation*}

(f): Given $m\colon M\to A$ from $\mathrm{Mono}_{SE}(\mathcal{C})$, a morphism $x\colon X\to A$, and a monomorphism $u\colon U\to X$, consider the diagram
	\begin{equation*}\xymatrix{M\times_AU\ar[ddd]_{1\times e}\ar[rd]_{1\times u}\ar[rr]&&U\ar[ddd]_e\ar[rd]_u\\&M\times_AX\ar[ddd]\ar[rr]&&X\ar[ddd]_x\\\\M\times_AN\ar[rr]\ar[rd]&&N\ar[rd]_n\\&M\ar[rr]_m&&A,}
	\end{equation*}
in which $M\times_AU=(M\times_AX)\times_XU,$ $xu=ne$ is a (regular epi, mono) factorization of $xu$, and the unlabeled arrows are the suitable pullback projections. Assuming $M\times_AU=0$, we have to prove that $U=0$. Indeed:
\begin{itemize}
	\item Since $e$ is a regular epimorphism, so is $1\times e$.
	\item Since $1\times e$ is an epimorphism and $M\times_AU=0$, we have $M\times_AN=0$.
	\item Since $M\times_AN=0$ and $m$ is in $\mathrm{Mono}_{SE}(\mathcal{C})$, we have that $N=0$.
	\item Since $N=0$, we have $xe=ne=0$, and so $u$ factors through the kernel of $x$.
	\item Since $u$ factors through the kernel of $x$, it also factors through the pullback projection $M\times_AX\to X$.
	\item Since $u$ factors through the pullback projection $M\times_AX\to X$, and the top part of our diagram is a pullback, the pullback projection $M\times_AU\to U$ is a split epimorphism.  
\end{itemize}
It follows that $U=0$, as desired.  
\end{proof}

\begin{remark} {\em For a composable pair $(m,m')$ of monomorphisms, $m'$ can be seen as a pullback of $mm'$ along $m$ (this well-known fact was used in the proof of 2.1(f) for the pair $(m,l)$). This implies that every pullback stable class of monomorphisms has the strong left cancellation property and, in particular, that, in the case of normal $\mathcal{C}$, \ref{subobj-essential}(e) could be deduced from \ref{subobj-essential}(f).}
\end{remark}

\begin{remark}\label{counter-weak-left} {\em In contrast to Proposition \ref{ST-Mono}(f) and \ref{subobj-essential}(e), the class $\mathrm{Mono}_E(\mathcal{C})$ does not even have, in general, the weak left cancellation property $m,\,mm'\in\mathrm{Mono}_E(\mathcal{C})\Rightarrow m'\in\mathrm{Mono}_E(\mathcal{C})$. One can easily construct counter-examples in many non-abelian semi-abelian algebraic categories by suitably choosing $m':M'\to M$ to be a split monomorphism (that is not an isomorphism) and choosing $m:M\to A$ with simple $A$. For example:
	\begin{itemize}
		\item [(a)] Let $\mathcal{C}$ be the category of groups, $A$ any simple group that has an element $a$ of order $pq$ with relatively prime $p$ and $q$, $M$ the subgroup of $A$ generated by $a$, $M'$ the subgroup of $A$ generated by $a^p$, and $m:M\to A$ and $m':M'\to M$ the inclusion maps. Then $m$ and $mm'$ are in $\mathcal{M}_E$, but $m'$ is not. 
		\item [(b)] Let $\mathcal{C}$ be the category of rings (commutative or not; we do not require them to have identity element, to make $\mathcal{C}$ semi-abelian), $M'=K$ be a field, $M=K[x]$ the polynomial ring in one variable $x$ over $K$, $A=K(x)$ the field of fractions of $M$, and $m:M\to A$ and $m':M'\to M$ the canonical monomorphisms. Then, again, $m$ and $mm'$ are in $\mathcal{M}_E$, but $m'$ is not.		
	\end{itemize}}
\end{remark}

\begin{remark}\label{Mono-Grp} {\em Although the non-pullback-stability of $\mathrm{Mono}_E(\mathcal{C})$ in the category of groups follows from \ref{counter-weak-left}(a), let us give what seems to be the simplest counter-example. Consider the pullback
\begin{equation*}\xymatrix{0\ar[d]\ar[r]&S_2\ar[d]\\A_3\ar[r]&S_3}
\end{equation*}
of monomorphisms, where $S_2$, $S_3$, $A_3$ are the symmetric/alternating groups. Its bottom arrow is an essential monomorphism, while the top one is not. This also shows that $A_3\to S_3$ is an example of an essential monomorphism that is not subobject-essential.}
\end{remark}
\begin{teo}\label{SE=St}
	If $\mathcal{C}$ is normal, then $\mathrm{Mono}_{SE}(\mathcal{C})=\mathrm{St}(\mathrm{Mono}_E(\mathcal{C}))$, that is, a morphism in $\mathcal{C}$ is a subobject-essential monomorphism if and only if it is a pullback stable essential monomorphism.
\end{teo}
\begin{proof}
	The inclusion $\mathrm{Mono}_{SE}(\mathcal{C})\subseteq\mathrm{St}(\mathrm{Mono}_E(\mathcal{C}))$ follows from \ref{coro-sub-essent}(a) and \ref{subobj-essential}(f). Conversely, let $m:M\to A$ be a pullback stable essential monomorphism in $\mathcal{C}$ and $n:N\to A$ a monomorphism in $\mathcal{C}$ with $M\times_AN=0$. Then $0\to N$ is an essential monomorphism because it is a pullback of $m$. Hence,  from the last assertion of \ref{Prop-Mono-ECS}(d), $0\to N$ is an isomorphism, that is, $N=0.$  
\end{proof}

Since every abelian category is normal, we easily get:

\begin{cor} If $\mathcal{C}$ is abelian then $\Spec(\mathcal{C})$ is the same as the spectral category of $\mathcal{C}$ in the usual sense (see \cite{[GO]} and \cite{[GZ]}). \end{cor}

\begin{proof} Having in mind Corollary~\ref{coro-sub-essent}(c), this follows from Theorem 6.9 and the description of the spectral category of an abelian category given in 2.5(e) of \cite[Chapter~I]{[GZ]}.\end{proof}

\section{Uniform objects}

Let us return to the general situation of Remark 4.2(a), where $\mathcal{M}$ is an arbitrary class of morphisms in $\mathcal{C}$, but let us assume that $\mathcal{C}$ is pointed and that

\begin{equation*}\xymatrix{x\in \mathcal{M}\Rightarrow \mathsf{Ker}(x)=0.}
\end{equation*}
As already observed, {in any \emph{normal} category, this property simply says that $\mathcal{M}$ is a class of monomorphisms. }
\begin{defi}{\rm
	An object $A$ of $\mathcal{C}$ is said to be {\em $\mathcal{M}$-uniform} if a morphism $x \colon X\rightarrow  A$ belongs to $\mathcal{M}$ whenever $X\not=0$ and $\mathsf{Ker}(x)=0$.}
\end{defi}
The term \emph{uniform} comes from module theory, where a non-zero module $M$ is called a uniform module if every non-zero submodule $N$ of $M$ is an essential submodule. Note that this term was also used in \cite{[AF]} for an analogue notion in the category of $G$-groups. 
\begin{prop}\label{uniform}
	Let $A$ and $B$ be $\mathcal{M}$-uniform objects in $\mathcal{C}$. Every non-zero morphism $A\to B$ in $\mathcal{C}[\mathcal{M}^{-1}]$ of the form  $P_{\mathcal{M}}(f)P_{\mathcal{M}}(x)^{-1}$ with $x$ in $\mathcal{M}$ is an isomorphism. 
\end{prop}

\begin{proof}
	Consider the diagram
	\begin{equation*}\xymatrix{&\mathsf{Ker}(f)\ar[d]_{\ker(f)}\\A&X\ar[l]_x\ar[r]^f&B,}
	\end{equation*}
	where $X$ is the domain of $x$ (the domain of $f$). 
	
	Suppose $\mathsf{Ker}(f)\not=0$. Since $A$ is $\mathcal{M}$-uniform and $x$ and $\ker(f)$ have zero kernels, the composite $\mathsf{Ker}(f)\to X\to A$ belongs to $\mathcal{M}$. As $x$ also belongs to $\mathcal{M}$, this implies that $P_{\mathcal{M}}(\ker(f))$ is an isomorphism, and we can write
	\begin{eqnarray*}P_{\mathcal{M}}(f)P_{\mathcal{M}}(x)^{-1} &=& P_{\mathcal{M}}(f)P_{\mathcal{M}}(\ker(f))P_{\mathcal{M}}(\ker(f))^{-1}P_{\mathcal{M}}(x)^{-1}\\
	& = &P_{\mathcal{M}}(f(\ker(f))P_{\mathcal{M}}(\ker(f))^{-1}P_{\mathcal{M}}(x)^{-1}\\ & = &P_{\mathcal{M}}(0)P_{\mathcal{M}}(\ker(f))^{-1}P_{\mathcal{M}}(x)^{-1}\\ & =& 0.	
	\end{eqnarray*}
	That is, we can suppose $\mathsf{Ker}(f)=0$. If so, then, since $B$ is $\mathcal{M}$-uniform, $f$ belongs to $\mathcal{M}$, which makes $P_{\mathcal{M}}(f)P_{\mathcal{M}}(x)^{-1}$ an isomorphism.   
\end{proof}

In order to state the next result, let us recall that a \emph{division monoid} is a non-trivial monoid $M$ with the property that the submonoid $U(M)$ of invertible elements is given by $U(M) = M \setminus \{ 0 \}$. We write 
${\End}_{\mathrm{Spec}(\mathcal{C})} (A)$ for the monoid of endomorphisms of an object $A$ in the spectral category ${\mathrm{Spec}(\mathcal{C})}$, where $\mathcal{C}$ is a normal category.
\begin{cor}
Let $\mathcal C$ be a normal category, and $A$ an $\mathcal M$-uniform object in $\mathcal C$ for $\mathcal M$ being the class of subobject-essential monomorphisms in $\mathcal{C}$. Then the monoid ${\End}_{\mathrm{Spec}(\mathcal{C})} (A)$ of endomorphisms of $A$ in ${\mathrm{Spec}(\mathcal{C})}$ is a division monoid.
 \end{cor}
 \begin{proof}
 This immediately follows from Proposition \ref{uniform}, by taking into account the fact that the class of subobject-essential monomorphisms coincides with the class of pullback-stable essential monomorphisms whenever $\mathcal C$ is a normal category (by Theorem \ref{SE=St}).
 \end{proof}
 This last result extends Lemma $5.4$ in \cite{[AF]}, where the base category $\mathcal C$ was the category of $G$-groups, to the general context of a normal category $\mathcal C$.


\begin{thebibliography}{99}

\bibitem{[AHS]} J. Ad\'amek, H. Herrlich, and G. E. Strecker, ``Abstract and concrete categories. The joy of cats'', Pure and Applied Math., John Wiley \& Sons, Inc., New York, 1990.

{ \bibitem{AHRT} J. Ad\'amek, H. Herrlich, J. Rosick\'{y} and W. Tholen, Injective hulls are not natural, Algebra Universalis 48, 2002, 374-388.}

\bibitem{[AF]} M. J. Arroyo Paniagua and A. Facchini, Category of $G$-groups and its spectral category, Comm. Algebra 45, 4, 2017, 1696--1710.

{ \bibitem{Ban} B. Banaschewski, Injectivity and essential extensions in equational classes of algebras, in ``Queen's papers in Pure and Applied Mathematics No. 25'' (Proceedings of the Conference on Universal Algebra, Kingston 1969) pp. 131-147. Kingston, Ontario, 1970.}

\bibitem{[B1]} J. B\'enabou, Introduction to bicategories, in ``1967 Reports of the Midwest Category Seminar'', Springer, Berlin, pp. 1--77.

\bibitem{[B2]} J. B\'enabou, Some geometric aspects of the calculus of fractions, The European Colloquium of Category Theory (Tours, 1994), Appl. Categ. Structures 4, 2, 1996, 139--165.

\bibitem{[CJKP]} A. Carboni, G. Janelidze, G. M. Kelly, and R. Par\'e, On localization and stabilization of factorization systems, Appl. Categ. Structures 5, 1997, 1--58.

\bibitem{Chennai}  A. Facchini, Injective modules, spectral categories, and applications. In ``Noncommutative rings, group rings, diagram algebras and their applications'', S. K. Jain and S. Parvathi Eds., Contemp. Math. 456, Amer. Math. Soc., Providence, RI, 2008, pp. 1--17.
 
 \bibitem{Fichtner} K. Fichtner, Eine Bermerkung \"uber Mannigfaltigkeiten universeller Algebren mit Idealen, Monatsh. d. Deutsch. Akad. d. Wiss. (Berlin) 12, 1970, 21--25. 
 
\bibitem{[Ga]} P. Gabriel, Des cat\'egories ab\'eliennes, Bull. Soc. Math. France, 90, 1962, 323--448.

\bibitem{[GO]} P. Gabriel and U. Oberst, Spektralkategorien und regul\"are Ringe im von-Neumannschen Sinn., Math. Z. 92, 1966, 389--395.

\bibitem{[GZ]} P. Gabriel and M. Zisman, ``Calculus of fractions and homotopy theory'', Springer, Berlin, 1967.

\bibitem{goodearl}
K.~R. Goodearl, ``Ring Theory. Nonsingular Rings and Modules'', Marcel Dekker, Inc., New York and Basel, 1976.



\bibitem{[JM]} G. Janelidze and S. Mac Lane, Private communication.

\bibitem{JMT} G. Janelidze, L. M\'arki and W. Tholen, Semi-abelian categories, J. Pure Appl. Algebra 168, 2002, 367--386.

\bibitem{JMU} G. Janelidze, L. M\'arki and A. Ursini, Ideals and clots in universal algebra and in semi-abelian categories, J. Algebra 307, 2007, 191-208.

\bibitem{[JT]} G. Janelidze and W. Tholen, Strongly separable morphisms in general categories, Theory Appl. Categ. 23, 7, 2010, 136--149.

\bibitem{[Ja]} Z. Janelidze, The pointed subobject functor, $3\times3$ lemmas, and subtractivity of spans, Theory Appl. Categ. 23, 11, 2010, 221--242.

\bibitem{[Mi]} B. Mitchell, ``Theory of categories'', Pure and Applied Math. 17, Academic Press, 1965.



\bibitem{[S]} H. Schubert, ``Categories'', Springer-Verlag, New York-Heidelberg, 1972.


{ \bibitem{Tholen} W. Tholen, Injective objects and cogenerating sets, J. Algebra 73, 1981, 139--155. }

\bibitem{[Y]} N. Yoneda, On Ext and exact sequences, J. Fac. Sci. Univ. Tokyo Sect. I, 8, 1960, 507--576.

\end{thebibliography}
\end{document}